\documentclass[a4paper,11pt]{amsart}
\usepackage{stmaryrd}
\usepackage{amsfonts}
\usepackage{amssymb}
\usepackage{amsmath}
\usepackage{mathrsfs}
\usepackage{amsmath,amssymb,amsthm,latexsym,amscd,mathrsfs}
\usepackage{indentfirst}
\usepackage{graphicx}
\usepackage{booktabs}
\usepackage{array}
\usepackage{chemarrow}
\usepackage{verbatim}
\newcommand{\ROM}[1]{\mathrm{\uppercase\expandafter{\romannumeral#1}}}
\theoremstyle{definition}

\newtheorem{thm}{Theorem}[section]
\newtheorem{lem}{Lemma}[section]

\newtheorem{rem}{Remark}[section]
\newtheorem{prop}{Proposition}[section]

\newtheorem{ack}{Acknowledgements}   
\makeatletter



\title[on focal submanifolds of isoparametric hypersurfaces and Simons formula]{\textbf{on focal submanifolds of isoparametric hypersurfaces and Simons formula}}

\date{}
\thanks {$\dagger$ The second author is the corresponding author. }
\author[Q. C. Li]{Qichao Li}
\address{School of Mathematical Sciences, Laboratory of Mathematics and Complex Systems, Beijing Normal
University, Beijing 100875, China} \email{qichaoli@mail.bnu.edu.cn}
\author[L. Zhang]{Li {Zhang}$^\dagger$}
\address{School of Mathematical Sciences, Laboratory of Mathematics and Complex Systems, Beijing Normal
University, Beijing 100875, China} \email{zhangli@mail.bnu.edu.cn}

\begin{document}

\maketitle
\begin{abstract}
The focal submanifolds of isoparametric hypersurfaces in spheres are all minimal Willmore submanifolds, mostly being $\mathcal{A}$-manifolds
in the sense of A.Gray but rarely Ricci-parallel (\cite{QTY},\cite{LY},\cite{TY3}). In this paper we study the geometry of the focal submanifolds via Simons formula. We show that all the focal submanifolds with $g\geq3$ are not normally flat by estimating the normal scalar curvatures. Moreover, we give a complete classification of the semiparallel submanifolds among the focal submanifolds.
\end{abstract}

\section{Introduction}
Isoparametric hypersurfaces in a unit sphere are hypersurfaces with constant principal curvatures.
They consist of a one-parameter family
of parallel hypersurfaces which laminates the unit sphere with two focal submanifolds at
the end. It is remarkable that the focal submanifolds of isoparametric hypersurfaces provide infinitely many spherical submanifolds with abundant intrinsic and extrinsic geometric properties. For instance, they are all minimal Willmore submanifolds in unit spheres and mostly $\mathcal{A}$-manifolds in the sense of A.Gray (\cite{Gra}), except for two cases of $g=6$ (cf. \cite{QTY},\cite{TY2},\cite{LY},\cite{Xie}). As is well known, an Einstein manifold minimally immersed in a unit sphere is Willmore, however the focal submanifolds are rarely Ricci-parallel, thus rarely Einstein  (cf. \cite{TY3}).

In this paper we study the geometry of the focal submanifolds in terms of Simons formula (see \cite{Sim},\cite{CdK}).
To state our results clearly, we firstly give some preliminaries on isoparametric hypersurfaces and their focal submanifolds in unit spheres.

A remarkable
result of M{\"u}nzner (\cite{Mun}) states that the number $g$ of distinct principal curvatures must be 1,2,3,4 or 6 by the arguments of differential topology. Moreover, the isoparametric hypersurfaces  can be represented as the regular level sets  in a unit sphere of
  some \emph{Cartan-M{\"u}nzner polynomial} $F$. By a \emph{Cartan-M{\"u}nzner polynomial}, we mean a homogeneous polynomial $F$ of degree $g$ on $\mathbb{R}^{n+2}$ satisfying
the so-called \emph{Cartan-M{\"u}nzner equations}:
\begin{equation}\label{cm}
\left\{ \begin{array}{ll}
|\nabla F|^2  =   g^2 r^{2g-2}, \qquad r = |x|, \\
 \Delta F  =  c r^{g-2}, \qquad c = g^2 (m_2-m_1)/2, \nonumber
\end{array}\right.
\end{equation}
where $\nabla F$ and $\Delta F$ are the gradient and Laplacian of
$F$ on $\mathbb{R}^{n+2}$.  The function
$f=F|_{\mathbb{S}^{n+1}}$ takes values in $[-1, 1]$. The level sets $f^{-1}(t)~(-1<t<1)$ gives the isoparametric hypersurfaces. In fact, if we order the principal curvatures $\lambda_1>\cdots>\lambda_g$
with multiplicities $m_1,\cdots,m_g$, then $m_i=m_{i+2}(\mathrm{mod}~g)$, in particular, all multiplicities are equal
when $g$ is odd. On the other hand, the critical sets $M_+:=f^{-1}(1)$ and $M_-:=f^{-1}(-1)$ are connected submanifolds with codimensions $m_1+1$ and $m_2+1$ in $\mathbb{S}^{n+1}$, called \emph{focal submanifolds} of this isoparametric family.

The isoparametric hypersurfaces $M^n$ with $g \leq3$ in $\mathbb{S}^{n+1}$ were classified by Cartan to be homogeneous (\cite{Car1},\cite{Car2}).
For $g = 6$, Abresch (\cite{Abr}) proved that $m_1=m_2= 1$ or $2$. Dorfmeister-Neher (\cite{DN}) and Miyaoka (\cite{Miy3}) showed they are homogeneous, respectively.
For  $g=4$,
all the isoparametric hypersurfaces are recently proved to be
OT-FKM type (\cite{OT},\cite{FKM}) or homogeneous with $(m_1,m_2)=(2, 2), (4, 5)$, except possibly for the unclassified case with $(m_1,m_2)=(7, 8)$ or $(8,7)$ (\cite{CCJ},\cite{Imm},\cite{Chi}).

Recently, Tang and Yan (\cite{TY3}) provided a complete determination for which focal
submanifolds of $g=4$ are Ricci-parallel except for the unclassified case $(m_1, m_2)=(7,8)$ or $(8,7)$. More precisely, they proved
\begin{thm}(\cite{TY3})\label{thmQTY}
{\itshape For the focal submanifolds of isoparametric
hypersurfaces in spheres with $g=4$,
we have
\begin{itemize}
\item[(i)] The $M_+$ of OT-FKM type is Ricci parallel if and only if $(m_1 ,m_2 ) = (2,1),(6,1)$, or it is diffeomorphic to $Sp(2)$ in the homogeneous case with $(m_1 ,m_2 ) = (4,3)$; While the $M_-$ of OT-FKM type is Ricci-parallel if and only if $(m_1 ,m_2 ) = (1,k)$.
\item[(ii)] For $(m_1,m_2)=(2,2)$, the one
           diffeomorphic to $\widetilde{G}_2(\mathbb{R}^5)$ is Ricci-parallel, while the other
           diffeomorphic to $\mathbb{C}P^3$ is not.
\item[(iii)]For $(m_1,m_2)=(4,5)$, both are not Ricci-parallel.
\end{itemize}
}
\end{thm}

Inspired by Tang and Yan's result, we give a complete classification of the \emph{semiparallel submanifolds} among all the focal submanifolds via Simons formula.

Recall that \emph{semiparallel submanifolds}
were introduced by Deprez in 1986 (see \cite{Dep}), as a generalization of \emph{parallel submanifolds}. Given an isometric immersion $f:M\rightarrow N$, denote by $B$  and $\bar{\nabla}$ its second
fundamental form and the connection in $(\mathrm{tangent~~bundle})\oplus(\mathrm{normal~~bundle})$, respectively, the immersion is said to be \emph{semiparallel} if
\begin{equation}\label{semiparallel submanifolds}(\bar{R}(X,Y)\cdot B)(Z,W):=(\bar{\nabla}^2_{XY}B)(Z,W)-(\bar{\nabla}^2_{YX}B)(Z,W)=0\end{equation}
 for any tangent
vectors $X, Y,Z ~\mathrm{and}~~ W$ to $M$, where $\bar{R}$ is the curvature tensor of the connection $\bar{\nabla}$ in $(\mathrm{tangent~~bundle})\oplus(\mathrm{normal~~bundle})$.
 An elegant survey
on the study of semiparallel submanifolds in a real space form can be found in \cite{Lum}.

Our main results state as follows:
\begin{thm}\label{thmzl}
{\itshape For focal submanifolds of isoparametric
hypersurfaces in unit spheres with $g\geq3$, we have
\begin{itemize}
\item[(i)]  For $g=3$: all the focal submanifolds are semiparallel.
\item[(ii)]  For $g=4$: the $M_+$ of OT-FKM type with $(m_1 ,m_2 ) = (2,1),(6,1)$, the $M_+$ of homogeneous OT-FKM type with $(m_1 ,m_2 ) = (4,3)$, all of the $M_-$ of OT-FKM type with $(m_1 ,m_2 ) = (1,k)$ and the focal submanifold diffeomorphic to $\widetilde{G}_2(\mathbb{R}^5)$ with $(m_1,m_2)=(2,2)$ are semiparallel. None of the others are semiparallel.
\item[(iii)] For $g=6$: none of the focal submanifolds are semiparallel.
\end{itemize}
}
\end{thm}

\begin{rem}
Our results work on the unclassified case of $g=4$, $(m_1,m_2)=(7,8)$ or $(8,7)$.
\end{rem}

This paper is organized as follows. In Section 2, we firstly state some preliminaries on geometry of submanifolds and estimate the normal scalar curvatures of focal submanifolds. In the next section, we proved Theorem \ref{thmzl} via Simons formula .

\section{Notations and Preliminaries}
Let $f:M^n\rightarrow \mathbb{S}^{n+p}$ be a compact $n$-dimensional minimal submanifold of the unit sphere.
 Denote by $B$ the second fundamental form, $\bar{\nabla}$ the
covariant derivative
with respect to the connection in $(\mathrm{tangent~~bundle})\oplus(\mathrm{normal~~bundle})$, $\nabla$ ($\nabla^\bot$) the
 induced Levi-Civita (normal) connection in the tangent (normal) bundle, respectively. Choosing a local field $\xi_1,\cdots,\xi_p$
of orthonormal frames of $T^\bot M$, we set $A_\alpha=A_{\xi_\alpha}$, the shape operator with respect to $\xi_\alpha$.

 The first covariant derivative of $B$ is defined by
     $$(\bar{\nabla}_XB)(Z,W)=(\bar{\nabla}B)(X,Z,W):=\nabla^\bot_XB(Z,W)-B(\nabla_XZ,W)-B(Z,\nabla_XW).$$
     By the Codazzi equation, $(\bar{\nabla}_XB)(Z,W)$ is symmetric in the tangent vectors $X,Z$ and $W$.

     Define the second covariant derivative of $B$ by
     $$(\bar{\nabla}_{XY}^2B)(Z,W):=(\bar{\nabla}_X(\bar{\nabla}B))(Y,Z,W)=(\bar{\nabla}(\bar{\nabla}B))(X,Y,Z,W),$$
     where $X,Y,Z,W\in TM.$ And the (rough) Laplacian of $B$ is defined by
     $$(\Delta B)(X,Y):=\sum_{i=1}^n(\bar{\nabla}_{e_ie_i}^2B)(X,Y),$$  where $e_1,\cdots,e_n$ is a local orthonormal frame of $TM.$

The famous Simons formula states as follows:
\begin{lem}\label{simonsssssssssssss}
(\cite{Sim} \cite{CdK}) Denote by $\|B\|^2$ the squared norm of the 2nd fundamental form of a submanifold $M^n$ minimally immersed in unit sphere $\mathbb{S}^{n+p}$, then
 \begin{eqnarray}\label{simons}
 \frac{1}{2}\Delta\|B\|^2
 &=&\|\bar{\nabla}B\|^2+\langle B,\Delta B\rangle\\\nonumber
 &=&\|\bar{\nabla}B\|^2+n\|B\|^2-\sum_{\alpha,\beta=1}^p\langle A_\alpha,A_\beta\rangle^2-\sum_{\alpha,\beta=1}^p\|[A_\alpha,A_\beta]\|^2,
 \end{eqnarray}
 where $\langle A_\alpha,A_\beta\rangle:=\mathrm{tr}(A_\alpha A_\beta)$, $\|[A_\alpha,A_\beta]\|^2:=-\mathrm{tr}(A_\alpha A_\beta-A_\beta A_\alpha)^2$ and $\Delta:=\mathrm{tr}(\nabla^2)$, the Laplace operator of $M^n$.
\end{lem}

The scalar curvature of the normal bundle is defined as
$\rho^\bot=\|R^\bot\|,$
where $R^\bot$ is the curvature tensor of the normal bundle. By the Ricci equation $\langle R^\bot(X,Y)\xi,\eta\rangle=\langle[A_\xi,A_\eta]X,Y\rangle$, the normal scalar curvature  can be rewritten as
\begin{equation}\label{normalllllllllll scalar curvature}\nonumber
\rho^\bot=\sqrt{\sum_{\alpha,\beta=1}^p\|[A_\alpha,A_\beta]\|^2}.
\end{equation}
Clearly, the normal scalar curvature is always nonnegative and the $M^n$ is normally flat if and only if $\rho^\bot=0$, and by a result of Cartan, which is equivalent to the simultaneous diagonalisation  of all shape operators $A_\xi.$

\begin{lem}
(\cite{CR})The focal submanifolds of isoparametric hypersurfaces are \emph{austere submanifolds} in unit spheres, and the shape operators of focal submanifolds are isospectral, whose principal curvatures are $\mathrm{cot}(\frac{k\pi}{g}),~1\leq k \leq g-1$ with  multiplicities $m_1$ or $m_2$, alternately.
\end{lem}

\begin{prop}\label{squared norm}
For focal submanifolds of isoparametric
hypersurfaces in $\mathbb{S}^{n+1}$ with $g\geq3$, the squared norms of the second fundamental forms $\|B\|^2$ are constants:
\begin{itemize}
\item[(i)]  For the cases of $g=3$ and multiplicities $m_1=m_2:=m,$
\begin{equation}\label{g3m1248}\nonumber
\|B\|_{M_\pm}^2=\frac{2}{3}m(m+1),  \quad m=1,2,4,8.
\end{equation}
\item[(ii)]  For the cases of $g=4$ and multiplicities $(m_1,m_2)$,
\begin{equation}\label{g4m1}\nonumber
\|B\|^2_{M_+}=2m_2(m_1+1).
\end{equation}
          Similarly,
\begin{equation}\label{g4m2}\nonumber
\|B\|^2_{M_-}=2m_1(m_2+1).
\end{equation}
\item[(iii)] For the cases of $g=6$,
\begin{equation}\label{g6m12}\nonumber
\|B\|_{M_\pm}^2=\frac{20}{3}m(m+1),  \quad m=1,2.
\end{equation}
\end{itemize}
\end{prop}
\begin{proof}
The proof follows easily from $\|B\|^2=\sum_{\alpha=1}^p \mathrm{tr}(A_\alpha^2)$ and Lemma 2.2.
\end{proof}

\begin{thm}\label{normal scalar curvature}
For focal submanifolds of isoparametric
hypersurfaces in $\mathbb{S}^{n+1}$ with $g\geq3$, the normal scalar curvatures are estimated as follows:
\begin{itemize}
\item[(i)]  For the cases of $g=3$ and multiplicities $m_1=m_2=:m$,
\begin{equation}\label{shanyao}\nonumber\rho_{M_{\pm}}^\bot=\frac{2}{3}m\sqrt{2(m+1)},\quad m\in\{1,2,4,8\}.\end{equation}
\item[(ii)]  For the cases of $g=4$ and multiplicities  $(m_1,m_2)$,
\begin{equation}\label{mzheng}\nonumber
\rho_{M_{+}}^\bot\geq\sqrt{2m_1m_2(m_1+1)},\quad \rho_{M_{-}}^\bot\geq\sqrt{2m_1m_2(m_2+1)}.
\end{equation}
\item[(iii)] For the cases of $g=6$ and multiplicities $m_1=m_2=:m$,
\begin{equation}\label{daashanyao}(\rho_{M_{+}}^\bot)^2=(72+\frac{8}{9})m^2(m+1),\quad (\rho_{M_{-}}^\bot)^2=\frac{80}{9}m^2(m+1)\quad m\in\{1,2\}.\end{equation}
\end{itemize}
As a result, all the focal submanifolds with $g\geq3$ are not normally flat.
\end{thm}

\begin{proof}
 By definition, the normal scalar curvature is given by
 \begin{eqnarray}\label{normal scalar curvature calculation}
(\rho^\bot)^2&=&\sum_{\alpha,\beta=1}^p\|[A_\alpha,A_\beta]\|^2=\sum_{\alpha,\beta=1}^p\mathrm{tr}(A_\alpha A_\beta-A_\beta A_\alpha)(A_\beta A_\alpha-A_\alpha A_\beta)\\\nonumber&=&2\mathrm{tr}\sum_{\alpha,\beta=1}^p(A_\alpha^2A_\beta^2)-2\mathrm{tr}\sum_{\alpha,\beta=1}^p(A_\alpha A_\beta)^2.
\end{eqnarray}
We estimate the normal scalar curvatures case by case.
\begin{itemize}
\item[(i)]  For the cases of $g=3$ and multiplicities $m_1=m_2=:m$, $n=2m$, $p=m+1.$
Let $\xi$ be an arbitrary unit normal vector of a focal submanifold, by Lemma 2.2, the shape operator
$A_\xi$ has two distinct eigenvalues, $\frac{1}{\sqrt{3}}$ and $-\frac{1}{\sqrt{3}}$, with the same multiplicities $m$.
It follows easily that \begin{equation}\label{xiaofeng}(A_\xi)^2=\frac{1}{3}\mathrm{Id}.\end{equation}
Since the shape operators defined on the unit normal bundle are isospectral, for any two orthogonal unit normal vectors $\xi$ and $\eta$,
 $A_{\frac{1}{\sqrt{2}}(\xi+\eta)}$ still satisfies
 \begin{equation}\nonumber A_{\frac{1}{\sqrt{2}}(\xi+\eta)}^2=\frac{1}{3}\mathrm{Id},\end{equation}that is
 \begin{equation}\nonumber\frac{1}{2}A_{\xi}^2+\frac{1}{2}A_{\xi}A_{\eta}+\frac{1}{2}A_{\eta}A_{\xi}+\frac{1}{2}A_{\eta}^2=\frac{1}{3}\mathrm{Id},\end{equation}
 hence
 \begin{equation}\nonumber A_{\xi}A_{\eta}+A_{\eta}A_{\xi}=0,\end{equation} which follows
 \begin{equation}\label{qiaofeng} \mathrm{tr}(A_{\xi}A_{\eta})^2=-\mathrm{tr}(A_{\eta}^2A_{\xi}^2)=-\mathrm{tr}(A_{\xi}^2A_{\eta}^2).\end{equation}
 Substituting (\ref{xiaofeng}) and (\ref{qiaofeng}) into (\ref{normal scalar curvature calculation}), one has
 \begin{eqnarray}\label{normal scalar curvature calculation 1}
(\rho_{M_{\pm}}^\bot)^2=8\sum_{1\leq\alpha<\beta\leq p}\mathrm{tr}(A_\alpha^2A_\beta^2)=8\cdot\frac{2m}{9}\cdot\frac{m(m+1)}{2}=\frac{8m^2(m+1)}{9}\nonumber.
\end{eqnarray}

\item[(ii)]  For the cases of $g=4$ and multiplicities  $(m_1,m_2)$:
Recall a formula of Ozeki-Takeuchi (see [OT] p.534 Lemma. 12. (ii))
           $$A_\alpha=A_{\beta}^{2} A_\alpha+A_{\alpha} A_{\beta}^{2}+A_{\beta} A_{\alpha} A_{\beta},\quad\alpha\neq\beta.$$
Multiply by $A_\beta$ on both sides,
          $$A_\alpha A_{\beta}=A_{\beta}^{2} A_\alpha A_{\beta}+A_{\alpha} A_{\beta}^3+A_{\beta} A_{\alpha} A_{\beta}^2,\quad\alpha\neq\beta.$$
Taking traces of both sides and observing that $A_\alpha^3=A_\alpha$ for a focal submanifold with $g=4$, we arrive at that
          $$\langle A_\alpha,A_\beta\rangle:=\mathrm{tr}(A_\alpha A_\beta)=0,\quad \alpha\neq\beta.$$
On the other hand,
          for a focal submanifold with $g=4$, says $M_+$, $\|B\|^2=2m_2(m_1+1)$ (by Proposition 2.1), $\langle A_\alpha,A_\alpha\rangle=2m_2$ (by Lemma 2.2).
Substituting these into Simons' formula (\ref{simons}), we conclude that
      \begin{eqnarray}\sum_{\alpha,\beta=1}^p|[A_\alpha,A_\beta]|_{M_+}^2&=&\|\bar{\nabla} B\|^2+2m_2(m_1+1)(2m_2+m_1)-4m_2^2(m_1+1)\nonumber\\
      &=&\|\bar{\nabla} B\|^2+2m_1m_2(m_1+1)\geq2m_1m_2(m_1+1).\nonumber
      \end{eqnarray}The calculations for $M_-$ are completely similarly.
\item[(iii)] For the cases of $g=6$ and multiplicities $m_1=m_2=:m$, all the focal submanifolds are homogeneous, and the result follows from a straightforward calculation and the explicit expressions for $A_\alpha$'s from Miyaoka (\cite{Miy1} \cite{Miy2}).
\end{itemize}
\end{proof}

\section{Semiparallel submanifolds and Simons formula}
For a semiparallel submanifold $M^n$ immersed in $\mathbb{S}^{n+p}$, denote by $e_1,\cdots,e_n$ a local orthonormal frame of $TM,$ $X,Y\in TM$, and $H$ the mean curvature vector of $M^n$, then
\begin{eqnarray}\label{h=0}(\Delta B)(X,Y)&:=&\sum_{i=1}^n(\bar{\nabla}_{e_ie_i}^2B)(X,Y)=\sum_{i=1}^n(\bar{\nabla}_{XY}^2B)(e_i,e_i)\\\nonumber
&=&\bar{\nabla}_{XY}^2(\sum_{i=1}^nB(e_i,e_i))=\bar{\nabla}_{XY}^2(H)=0,
\end{eqnarray}
where the second equality follows from the definition of semiparallel submanifolds and Codazzi equation, and the third from the commutativity of the covariant differentiations and taking trace. By the arbitrariness of $X,Y\in TM$, one concludes that $\Delta B=0$ for a minimal semiparallel submanifold.

On the other hand, for a focal submanifold of an isoparametric hypersurface, $\|B\|^2=\mathrm{const.}$ (by Proposition 2.1), $H=0$ (by Lemma 2.2), combining with the Simons formula $\frac{1}{2}\Delta\|B\|^2=\|\bar{\nabla}B\|^2+\langle B,\Delta B\rangle$ one arrives at that $\|\bar{\nabla}B\|^2+\langle B,\Delta B\rangle=0$, thus a focal submanifold of an isoparametric hypersurface is semiparallel if and only if $\|\bar{\nabla}B\|^2=0$, namely, it is a parallel submanifold. It remains to give a complete classification of the focal submanifolds satisfying $\|\bar{\nabla}B\|^2=0$.

It is convenient to define the covariant derivative $(\bar{\nabla}_{X}A)_{\alpha}$ of $B$ along $\xi_\alpha\in U^\bot M$
by $\langle(\bar{\nabla}_{X}A)_{\alpha}Y,Z\rangle=\langle(\bar{\nabla}B)(X,Y, Z),\xi_\alpha\rangle.$
It follows easily that
$$(\bar{\nabla}_{X}A)_{\alpha}Y=(\nabla_{X}A_{\alpha})Y-\sum_{\beta=1}^{p}s_{\alpha\beta}(X)A_{\beta}(Y),$$
where $s_{\alpha\beta}$ is the normal connection defined by $\nabla_{X}^{\perp}\xi_{\alpha}= \sum_{\beta=1}^ps_{\alpha\beta}(X)\xi_\beta.$ Obviously, $\|\bar{\nabla}B\|^2=\sum_{\alpha=1}^{p}\|(\bar{\nabla}A)_\alpha\|^2$, thus
\begin{equation}\label{zhangli}\|\bar{\nabla}B\|^2=0~~\Leftrightarrow~~(\bar{\nabla}_XA)_\alpha=0 ~~ \mathrm{for} ~~ \mathrm{all} ~~~ \alpha=1,\cdots,p.\end{equation}

The Ricci curvature tensor can be derived from the Gauss equation and the minimality of the focal submanifolds as
$\mathrm{Ric}=(n-1)\mathrm{Id}-\sum_{\alpha=1}^pA_\alpha^2.$
Moreover, by the fact $s_{\alpha\beta}+s_{\beta\alpha}=0$, one has $$\nabla_X\mathrm{Ric}=-\sum_{\alpha=1}^p(\nabla_{X}A)_{\alpha}A_{\alpha}-\sum_{\alpha=1}^pA_{\alpha}(\nabla_{X}A)_{\alpha}=-\sum_{\alpha=1}^p
(\bar{\nabla}_{X}A)_{\alpha}A_{\alpha}-\sum_{\alpha=1}^pA_{\alpha}(\bar{\nabla}_{X}A)_{\alpha},$$ thus
$\|\bar{\nabla}B\|^2=0$ implies ${\nabla}\mathrm{Ric}=0$ by (\ref{zhangli}), namely, Ricci-parallel. Putting all the facts above together, we conclude that
\begin{lem}\label{semiparallel ricci parallel}
A focal submanifold of an isoparametric hypersurface is semiparallel if and only if it is a parallel submanifold. In particular, a semiparallel focal submanifold must be Ricci-parallel.
\end{lem}

\noindent\textbf{\textbf{\emph{Proof of Theorem 1.2.}}}  \quad

     \noindent\textbf{(1) For the cases $g=3$}, $n=2m,~~p=m+1$, we have $$\|B\|^2=\frac{2}{3}m(m+1), \quad \sum_{\alpha,\beta=1}^p\|[A_\alpha,A_\beta]\|^2=\frac{8m^2(m+1)}{9},
      \quad \langle A_\alpha,A_\beta\rangle=\frac{2m}{3}\delta_{\alpha\beta}.$$
      Substituting these into Simons formula (\ref{simons}), one has
      \begin{eqnarray}\|\bar{\nabla}B\|_{M_{\pm}}^2&=&\sum_{\alpha,\beta=1}^p\langle A_\alpha,A_\beta\rangle^2+\sum_{\alpha,\beta=1}^p\|[A_\alpha,A_\beta]\|^2-n\|B\|^2\nonumber\\
      &=&\frac{(4m^2)(m+1)}{9}+\frac{8m^2(m+1)}{9}-\frac{4}{3}m^2(m+1)\nonumber\\
      &=&0,\nonumber
      \end{eqnarray}
     thus all the focal submanifolds with $g=3$ are semiparallel by Lemma 3.1.

\begin{rem}
Each focal submanifold with $g=3$ must be one of the Veronese embeddings of projective
planes $\mathbb{F}P^2$ into $\mathbb{S}^{3m+1}$,
where ${\mathbb{ F}}$ is the division algebra
$\mathbb{R}$, ${\mathbb{ C}}$, ${\mathbb{ H}}$,
or ${\mathbb{O}}$ for $m=1,2,4$, or $8,$ respectively. In facts, all the standard Veronese embeddings of compact symmetric spaces of rank one into unit spheres are known to be parallel submanifolds \cite{Sak}. Here we have provided a new proof for the cases of $\mathbb{F}P^2$, $\mathbb{F}=\mathbb{R},\mathbb{C},\mathbb{H}~~\mathrm{or}~~\mathbb{O}$.
\end{rem}

     \noindent\textbf{(2) For the cases $g=4$} and multiplicities $(m_1,m_2)$, for the focal submanifolds,
     say $M_{+}$, $n=2m_2+m_1,~~p=m_1+1$, we have \begin{equation}\label{to simons formula}\|B\|^2_{M_+}=2m_2(m_1+1),\quad \langle A_\alpha,A_\beta\rangle_{M_+}={2m_2}\delta_{\alpha\beta}.\end{equation}
     Recall a formula of Ozeki-Takeuchi (see [OT], p.534, Lemma. 12. (ii))
           $$A_\alpha=A_{\beta}^{2} A_\alpha+A_{\alpha} A_{\beta}^{2}+A_{\beta} A_{\alpha} A_{\beta},\quad \alpha\neq\beta.$$
Multiply by $A_\alpha$  and take traces on both sides,
          $$\mathrm{tr}(A_\alpha^2)=2\mathrm{tr}(A_\alpha^2A_{\beta}^{2})+\mathrm{tr}(A_{\alpha}A_{\beta})^2,$$ then
          $$m_1\|B\|^2=m_1(m_1+1)\mathrm{tr}(A_\alpha^2)=2\sum_{\alpha\neq\beta}\mathrm{tr}(A_\alpha^2A_{\beta}^{2})+\sum_{\alpha\neq\beta}\mathrm
          {tr}(A_{\alpha}A_{\beta})^2.$$ On the other hand, by the formula $(A_\alpha)^3=A_\alpha$ for $1\leq\alpha\leq p$ (see [OT] p.534 Lemma. 12. (i)),
           $$3\|B\|^2=2\sum_{\alpha=\beta}\mathrm{tr}(A_\alpha^2A_{\beta}^{2})+\sum_{\alpha=\beta}\mathrm
          {tr}(A_{\alpha}A_{\beta})^2,$$
          thus
           \begin{equation}\label{OTlemma12}(m_1+3)\|B\|^2=2\sum_{\alpha,\beta=1}^p\mathrm{tr}(A_\alpha^2A_{\beta}^{2})+\sum_{\alpha,\beta=1}^p\mathrm
          {tr}(A_{\alpha}A_{\beta})^2.\end{equation}
          Substituting (\ref{OTlemma12}) into (\ref{normal scalar curvature calculation}), one concludes that for $M_+$
          \begin{eqnarray}\label{estimate normal scalar curvature calculation}
\sum_{\alpha,\beta=1}^p\|[A_\alpha,A_\beta]\|_{M_+}^2=6\mathrm{tr}(\sum_{\alpha=1}^pA_\alpha^2\sum_{\beta=1}^pA_\beta^2)-2(m_1+3)\|B\|^2.
\end{eqnarray}
Similarly, for $M_-$,\begin{eqnarray}\label{estimate normal scalar curvature calculation M-}\nonumber
\sum_{\alpha,\beta=1}^p\|[A_\alpha,A_\beta]\|_{M_-}^2=6\mathrm{tr}(\sum_{\alpha=1}^pA_\alpha^2\sum_{\beta=1}^pA_\beta^2)-2(m_2+3)\|B\|^2.
\end{eqnarray}
Moreover, substituting (\ref{to simons formula}) and (\ref{estimate normal scalar curvature calculation}) into (\ref{simons}), one arrives at
 \begin{eqnarray}\label{simons again}
\|\bar{\nabla}B\|_{M_+}^2&=&4m_2^2(m_1+1)+6\mathrm{tr}(\sum_{\alpha=1}^pA_\alpha^2\sum_{\beta=1}^pA_\beta^2)-2(m_1+3)\|B\|^2\nonumber\\& &-(2m_2+m_1)\|B\|^2\\
&=&6\mathrm{tr}(\sum_{\alpha=1}^pA_\alpha^2\sum_{\beta=1}^pA_\beta^2)-6m_2(m_1+1)(m_1+2).\nonumber
 \end{eqnarray}
 Similarly,
  \begin{eqnarray}\label{simons again again}
\|\bar{\nabla}B\|_{M_-}^2=6\mathrm{tr}(\sum_{\alpha=1}^pA_\alpha^2\sum_{\beta=1}^pA_\beta^2)-6m_1(m_2+1)(m_2+2).
 \end{eqnarray}
We firstly show that the focal submanifolds of the cases $g=4,$ $(m_1,m_2)=(7,8)$ or $(8,7)$ can not be semiparallel. In fact, by the inequality $\mathrm{tr}(A^2)\geq\frac{(\mathrm{tr}A)^2}{n}$
for symmetric matrices,
\begin{eqnarray}\label{simons againsssss}
\|\bar{\nabla}B\|_{M_+}^2&\geq&\frac{6}{m_1+2m_2}(\mathrm{tr}\sum_{\alpha=1}^pA_\alpha^2)^2-6m_2(m_1+1)(m_1+2)\nonumber\\
&=&\frac{24}{m_1+2m_2}m_2^2(m_1+1)^2-6m_2(m_1+1)(m_1+2)\nonumber\\
&=&\frac{6m_1m_2(m_1+1)}{m_1+2m_2}(2m_2-m_1-2),\nonumber
 \end{eqnarray}
 Clearly, in  the cases of $(m_1,m_2)=(7,8)$ or $(8,7)$, one has $\|\bar{\nabla}B\|_{M_+}^2>0$, $\|\bar{\nabla}B\|_{M_-}^2>0$. By Lemma \ref{semiparallel ricci parallel}, the focal submanifolds of the cases $g=4,$ $(m_1,m_2)=(7,8)$ or $(8,7)$ are not semiparallel.

 We nextly show the Ricci-parallel focal submanifolds listed in Theorem \ref{thmQTY} are all semiparallel case by case.

 For the $M_-$ of OT-FKM type with $(m_1,m_2)=(1,k)$, which can be characterized as $M_-=\{(x\mathrm{cos}\theta,x\mathrm{sin}\theta)|x \in\mathrm{S}^{k+1}(1),{\theta}\in\mathrm{S}^1(1)\} =\mathrm{S}^1(1)\times\mathrm{S}^{k+1}(1)/(\theta,x)\sim(\theta+\pi,-x)$ (see \cite{TY1} Proposition 1.1).   Clearly, the Ricci tensor of $M_-$ can be characterized as $\mathrm{Ric}=0\oplus k\cdot\mathrm{Id}_{(k+1)\times (k+1)}.$ However, by the Gauss equation $\mathrm{Ric}=(k+1)\mathrm{Id}-\sum_{\alpha=1}^{k+1}\mathrm{A}_\alpha^2$, one has
 \begin{equation}\label{1km-}\sum_{\alpha=1}^{k+1}\mathrm{A}_\alpha^2=(k+1)\oplus\mathrm{Id}_{(k+1)\times (k+1)}.\end{equation}
Substituting (\ref{1km-}) into (\ref{simons again again}), one arrives at
    \begin{eqnarray}\label{simons again again once again}\nonumber
\|\bar{\nabla}B\|_{M_-^{k+2}}^2=6((k+1)^2+(k+1))-6(k+1)(k+2)=0.
 \end{eqnarray}
    According to Lemma 3.1, $M_-$'s of OT-FKM type with $(m_1,m_2)=(1,k)$ are semiparallel submanifolds.  Furthermore, the
isoparametric families of OT-FKM type with multiplicities $(2,1)$, $(6,1)$ are congruent to those
with multiplicities $(1,2)$, $(1,6)$, respectively (\cite{FKM}), thus $M_+$'s of OT-FKM type with $(m_1 ,m_2 ) = (2,1)~~\mathrm{or}~~(6,1)$ are also semiparallel submanifolds.

    The $M_+$ of OT-FKM type with $(m_1,m_2)=(4,3)$ (the homogeneous case), as proved in (\cite{QTY}), is an Einstein manifold. In this way,
    \begin{equation}\label{43m+qixing}\sum_{\alpha=1}^{p}\mathrm{A}_\alpha^2=\frac{\|B\|^2}{10}\cdot\mathrm{Id}_{
    10\times10}=3\cdot\mathrm{Id}_{10\times10}.\end{equation}
    Substituting (\ref{43m+qixing}) into (\ref{simons again}), one arrives at
     \begin{eqnarray}\label{simons again twice}\nonumber
\|\bar{\nabla}B\|_{M_+^{10}}^2=6\cdot9\cdot10-6\cdot3\cdot5\cdot6=0.
 \end{eqnarray}
  According to Lemma 3.1, $M_+$ of OT-FKM type with $(m_1,m_2)=(4,3)$ (the homogeneous family) is semiparallel submanifold.

    The focal submanifold diffeomorphic to $\widetilde{G}_2(\mathbb{R}^5)$ in the exceptional homogeneous case with $(m_1,m_2)=(2,2)$, as proved in (\cite{QTY}), is an Einstein manifold. In this way,
    \begin{equation}\label{22m+-qixing}\sum_{\alpha=1}^{p}\mathrm{A}_\alpha^2=\frac{\|B\|^2}{6}\cdot\mathrm{Id}_{
    6\times6}=2\cdot\mathrm{Id}_{6\times6}.\end{equation}
    Substituting (\ref{22m+-qixing}) into (\ref{simons again}), one arrives at
     \begin{eqnarray}\label{simons 111 again twice third times}\nonumber
\|\bar{\nabla}B\|_{M^{10}}^2=6\cdot4\cdot6-6\cdot2\cdot3\cdot4=0.
 \end{eqnarray}
  According to Lemma 3.1, the focal submanifold diffeomorphic to $\widetilde{G}_2(\mathbb{R}^5)$ in the exceptional homogeneous case with $(m_1,m_2)=(2,2)$ is semiparallel submanifold.

  The proof of the second part of Theorem 1.2 is now complete.

 \noindent\textbf{(3) For the cases $g=6$}, $n=5m, p=m+1$, we have
 $$\|B\|^2=\frac{20}{3}m(m+1),\quad \langle A_\alpha,A_\beta\rangle=\frac{20m}{3}\delta_{\alpha\beta}.$$
      Substituting these and (\ref{daashanyao}) into Simons formula (\ref{simons}), one has
      \begin{eqnarray}\|\bar{\nabla}B\|_{M_+}^2&=&\sum_{\alpha,\beta=1}^p\langle A_\alpha,A_\beta\rangle^2+\sum_{\alpha,\beta=1}^p\|[A_\alpha,A_\beta]\|^2-n\|B\|^2\nonumber\\
      &=&\frac{(400m^2)(m+1)}{9}+(72+\frac{8}{9})m^2(m+1)-\frac{100}{3}m^2(m+1)\nonumber\\
      &=&84m^2(m+1)>0. \nonumber
      \end{eqnarray}

      Similarly,\begin{equation}\nonumber\|\bar{\nabla}B\|_{M_-}^2=20m^2(m+1)>0.\end{equation}
     Hence, none of the focal submanifolds with $g=6$ are semiparallel by Lemma 3.1.

The proof of Theorem \ref{thmzl} is now complete.
$\hfill{\square}$

\begin{ack}
The authors are greatly indebted to Professor Zizhou Tang for his valuable suggestions and encouragement.
\end{ack}

\end{document}